\documentclass[letter]{article}

\usepackage{amsthm,stmaryrd, amsfonts, mathrsfs,amssymb, amsmath,bbm}  
\usepackage[numbers, sort&compress]{natbib} 
\usepackage{vmargin, enumerate, paralist} 
\usepackage{times} 
\usepackage[latin1]{inputenc}
\usepackage[pdftex]{hyperref}

\newtheorem{thm}{Theorem}
\newtheorem{lem}[thm]{Lemma}
\newtheorem{prop}[thm]{Proposition}

\newcommand{\N}{\ensuremath{\mathbb N}}
\newcommand{\C}{\ensuremath{\mathbb{C}}}

\newcommand{\cF}{\mathcal F}

\newcommand{\E}[1]{\ensuremath{\mathbf{E} \left[#1 \right]}}
\newcommand{\Prob}[1]{\ensuremath{\mathbf{P} \left(#1 \right)}}
\newcommand{\I}[1]{\ensuremath{\mathbf{1}_{ \{ #1 \} }}}

\hypersetup{
    bookmarks=true,         
    unicode=false,          
    pdftoolbar=true,        
    pdfmenubar=true,        
    pdffitwindow=true,      
    pdftitle={My title},    
    pdfauthor={Author},     
    pdfsubject={Subject},   
    pdfnewwindow=true,      
    pdfkeywords={keywords}, 
    colorlinks=true,       
    linkcolor=blue,          
    citecolor=blue,        
    filecolor=blue,      
    urlcolor=blue           
}

\begin{document}

\title{\bf  On martingale tail sums for the path length in random trees
}
\author{ Henning Sulzbach\thanks{School of Computer Science, McGill University, 3480 University Street, Montreal, Canada, H3A 0E9. \newline
Email: henning.sulzbach@gmail.com \newline
This work was supported by a Feodor Lynen Research Fellowship from the Alexander
von Humboldt Foundation.
}
} 

\maketitle

\begin{abstract}
For a martingale $(X_n)$ converging almost surely to a random variable $X$, the sequence $(X_n - X)$ is called martingale tail sum. 
Recently, Neininger [\emph{Random Structures Algorithms}, 46 (2015), 346-361] proved a central limit theorem for the martingale tail sum of R{\'e}gnier's martingale for the path length in random binary search trees.  Gr{\"u}bel and Kabluchko [\emph{to appear in Annals of Applied Probability}, (2016)]
gave an alternative proof also conjecturing a corresponding law of the iterated logarithm.
We prove the central limit theorem with convergence of higher moments and the law of the iterated logarithm for a family of trees containing binary search trees, recursive trees and plane-oriented recursive trees. 
 \end{abstract}


\section{Introduction and main results}
In a finite rooted tree, the path length denotes the sum of the depths of all nodes. Here, the depth of a node equals its graph distance to the root. 
When the underlying tree is random, one is interested in its
average behavior,  in normal and large deviations, as well as in distributional (or almost sure) limit theorems as the size, that is the number of nodes, tends to infinity. 
In this paper we consider random trees whose heights grow logarithmically through their sizes, prominent examples are the binary search tree (BST), the recursive tree (RT) and the plane-oriented recursive tree (PORT).   
Trees of this kind are of prime importance in computer science where they serve as data structures and play fundamental roles in the analysis of algorithms, see, e.g.\ Mahmoud's book \cite{mahmoudbook} for a general account. Further, 
PORTs belong to the family of graphs of preferential attachment type, an important class of graphs in modeling  phenomena in real-world networks. 

\medskip \textsc{Results in the literature.} 
The path length $P_n$ in a binary search tree of size $n$ was first investigated by Hoare \cite{hoare} in his seminal paper on the quicksort algorithm.  Under the common probabilistic model (see the next paragraph for details), he obtained an exact expression for the mean of $P_n$. 
The variance of $P_n$ was calculated by Knuth \cite{Knuth1973b}. R{\'e}gnier \cite{regnier} showed that the sequence $(P_n - \E{P_n})/(n+1)$  is an $L_2$-bounded martingale, hence almost surely convergent. Shortly after, R{\"o}sler \cite{Ro91} invented the so-called contraction method to obtain the same convergence result (on a distributional level), further characterizing the limiting distribution as the solution to a stochastic fixed-point equation. 
In RTs and PORTs, the martingale approach for the path length was worked out by Mahmoud \cite{mahrec,mahorient}. Dobrow and Fill \cite{DobFill99} obtained analogues of R{\"o}sler's result in RTs. In BSTs, Fill and Janson \cite{fillja} and Neininger and R{\"u}schendorf \cite{neirueexact} investigated distances between the distribution of the rescaled quantity and its limit with respect to several probability metrics.
Neininger and R{\"u}schendorf \cite{neiruesplit} formulated a distributional limit theorem  for the class of split trees introduced by Devroye \cite{devsplit} assuming a suitable expansion of the mean of $P_n$. In the special case of higher-dimensional binary search trees, so-called quad trees, this expansion had already been proved by Flajolet et al. \cite{flhyper}. In the general case, it was established recently by Broutin and Holmgren \cite{Brholm} and independently by Munsonius \cite{Munso}.
Finally, for the family of weighted $b$-ary trees introduced by Broutin and Devroye \cite{DeBrweighted}, the path length was analyzed by Munsonius and R{\"u}schendorf \cite{Munso2}. 
In recent years, deeper results on the profile of random trees, that is on the number of nodes on a given level (i.e.\ of a given depth), have been obtained. These results also lead to limit theorems for the path length. We refer to Section 6 in Drmota's book \cite{drmotabook} for an overview of this development. We exploit the connection between the profile and the path length in our work and give further references in this context in the discussion of \eqref{defWW} and in Section \ref{sec:poly}.

\medskip \textsc{Probabilistic model.}  We consider the one-parametric family of linear recursive trees introduced by Pittel \cite{Pit94}. Let $\beta \in [0, \infty)$. A sequence of rooted trees $(T_n)_{n \geq 1}$ where $T_n$ has size $n$ is constructed as follows: The tree $T_1$ consists of a single node, the root. For $n \geq 2$, given a tree $T_{n-1}$ of size $n-1$, a node $v \in T_{n-1}$ is chosen with probability proportional to $\beta d_v + 1$. Here, $d_v$ denotes the number of children (i.e.\ the outdegree) of $v$. Then, $T_n$ is obtained by connecting a new node to $v$.
Moreover, we consider so-called $m$-ary trees constructed in the same way, where a new node is connected to an existing node $v$ with probability proportional to $m - d_v$, $1 < m \in \N$. (Note that, by construction, $d_v \leq m$ for all nodes $v$ at all times.) In order to simplify notation, we cover both models by writing $\beta d_v + m$ for the probability that $v$ is chosen as the parent node distinguishing the cases that either $\beta \in [0, \infty), m = 1$ or $\beta = -1, 1 < m \in  \N$. The table below shows important tree models which are covered by this setting.

Note that, if $\beta$ is integer-valued, then, to each node $v \in T_n$ with depth $k \geq 0$, one can associate $\beta d_v + m$ so-called \emph{external} nodes on level $k+1$.
The transition from $T_{n-1}$ to $T_n$ corresponds to replacing a randomly chosen external node in $T_{n-1}$  by an (\emph{internal}) node. \begin{center}
\begin{tabular}{l|c |c}
Tree model & $\beta$ & $m$ \\
\hline
Recursive tree & $0$ &  $1$ \\
Plane-oriented recursive tree & $1$ & $1$ \\
$p$-oriented tree, $p \in \N$ & $p$ & $1$ \\
Binary search tree & $-1$ & $2$ \\
$m$-ary tree, $1 < m  \in \N$ & $-1$ & $m$
\end{tabular}
\end{center}

\medskip \textsc{Mean and variance.} Before describing the main purpose of our work, we shortly discuss mean and variance of the path length. As a by-product of our analysis, compare \eqref{ding4} and the discussion of \eqref{ding3} below, we obtain the following asymptotic expansions, 
\begin{align} \label{exp:mean}
\E{P_n} = \frac{m}{\beta + m} n \log n -  \frac{m}{\beta + m} \left(1 + \psi^{(0)} \left(\frac{m} {\beta + m}\right)\right) n + O(\log n), \end{align}
and 
\begin{align} \label{exactvar} \text{Var}(P_n) =  \left( 1 + \frac{m}{\beta + m}\left( 1 - \frac{m}{\beta + m} \psi^{(1)} \left(\frac{m} {\beta + m}\right)\right) \right) n^2 + o(n^2). \end{align}
Here $\psi^{(0)} = (\log \Gamma)', \psi^{(1)} = (\log \Gamma)'',$ denote the digamma and trigamma functions. For BSTs \cite{hoare, Knuth1973b, regnier, Ro91}, RTs \cite{mahrec, DobFill99} and PORTs \cite{mahorient}, these results are classical. For all remaining cases of integer-valued $\beta$, \eqref{exp:mean} was obtained in \cite{Munso2}. For $\beta = - 1, 2 < m \in \N$ or $\beta \in \N_0, m = 1$, the order of the variance was known \cite{neiruesplit, Brholm, Munso, Munso2}, and the leading constant can be computed from the results in these works. Both expansions are novel for non-integer $\beta$.

\medskip \textsc{Object of study.} Our work is primarily concerned with the martingale approach relying on the sequence 
\begin{align} \label{defSn}
S_n := \frac{P_n - \E{P_n}}{n - \beta / (\beta + m)}. 
\end{align}
In Section \ref{sec:pre}, see also Theorem \ref{thm_main}, we verify the martingale property of $S_n$ which constitutes the starting point of our analysis. By \eqref{exactvar}, $S_n$ converges almost surely and in $L_2$ to a random variable $S$. It is natural to investigate the behavior of the martingale tail sum $S - S_n$. 
In random BSTs, an explicit formulae for its second moment was derived by Bindjeme and Fill \cite{bindfill}. Neininger \cite{ralphrefined} was first to prove asymptotic normality of $S_n - S$ when properly rescaled: in distribution, as $n \to \infty$, 
\begin{align} \label{CLT}
\sqrt{\frac{n}{2\log n}} (S_n - S) \to \mathcal{N}. \end{align}
Here, and throughout the work, we denote by $\mathcal N$ a generic random variable with the standard normal distribution.
A proof of this result based on the method of moments was worked out shortly after by Fuchs \cite{fuchs}.  Very recently, Gr{\"u}bel and Kabluchko \cite{grza} proved more general functional limit theorems in branching random walks containing \eqref{CLT} for a stronger mode of convergence  \cite[Section 5.5.1]{grza}. Moreover, they conjectured a corresponding law of the iterated logarithm. 

The purpose of this paper is to obtain Gaussian limit laws of type \eqref{CLT}  and laws of the iterated logarithm for trees in the model introduced above.

\medskip \textsc{Main results.}
Our main result is the following theorem. Here, and in the remaining of the paper, we define  $S_0 := 0$, $\mathcal{F}_0 := \{ \emptyset, \Omega\}$ and, for $n \geq 1$, $\mathcal{F}_n = \sigma(T_1, \ldots, T_n)$.

\begin{thm} \label{thm_main}
Let either $\beta \in [0, \infty), m = 1$ or $\beta = -1$ and $1 < m \in \N$.  Then, $S_n$ defined in \eqref{defSn} is an $L_2$-bounded martingale with respect to the filtration $(\cF_n)$. Moreover, denoting  by $S$ the almost sure limit of $S_n$, in distribution, 
\begin{align} \label{CLTtheorem}
\sqrt{\frac{\beta + m}{m}} \sqrt{\frac{n}{\log n }}  (S_n - S) & \to \mathcal{N}.
\end{align}
Almost surely, 
\begin{align*}
\limsup_{n \to \infty} \sqrt{\frac{\beta + m}{2m}} \sqrt{\frac{n}{\log n \log \log n}}  (S_n - S) & = 1, \\
\liminf_{n \to \infty} \sqrt{\frac{\beta + m}{2m}} \sqrt{\frac{n}{\log n \log \log n}}  (S_n - S) & = -1. 
\end{align*}
\end{thm}
The following result about convergence of higher moments in \eqref{CLTtheorem} presumably holds for all trees in our model. However, our proof below relies on the fact that $\beta$ is integer-valued.
\begin{thm} \label{thm:moments}
Let either $\beta \in \N_0, m = 1$ or $\beta = -1, 1 < m \in \N$. Then, the convergence in \eqref{CLTtheorem} is with respect to all moments. In other words, for $p  > 0$ and as $n \to \infty$, 
$$\left( \frac{n}{\log n}\right)^{p/2} \E{|S_n - S|^p} \to \left( \frac{m}{\beta + m}\right)^{p/2} \E{ |\mathcal{N}|^p}.$$
\end{thm}

Both theorems follow from general results in the context of martingale limit theorems. 

\medskip \textsc{A martingale limit theorem.} In the following theorem and in the remaining of the paper, we say that a sequence of real-valued random variables $Y_n, n \geq 0,$ is bounded in $L_p$, $p> 0$, if $\sup_{n \geq 0} \E{|Y_n|^p} < \infty$.
\begin{prop} \label{prop:mart}
Let $Z_n, n \geq 0,$ be a zero-mean, $L_2$-bounded martingale with respect to a filtration $\mathcal{G}_n, n \geq 0$. For $n \geq 1$, let $\Delta_n = Z_n - Z_{n-1}$ and $s_n^2 = \sum_{i=n}^\infty \E{\Delta_i^2}$. Assume that $s^2_n > 0$ for all $n$ and denote by $Z$ the almost sure limit of $Z_n$.
 If
\begin{itemize}
\item [\textbf{C1.}] $s_n^{-2} \sum_{i=n}^\infty  \E{\Delta_i^2 \I{|\Delta_i| \geq \varepsilon s_n}} \to 0$ for all $\varepsilon > 0$, and
\item [\textbf{C2.}] $s_n^{-2} \sum_{i = n}^\infty \E{\Delta_i^2 | \mathcal{G}_{i-1}} \to 1$ in probability, 
\end{itemize}
then \begin{align} \label{conv1} s_n^{-1} (Z_{n-1} - Z) \to \mathcal {N} \end{align} in distribution. If
\begin{itemize}
\item [\textbf{L1.}] $\sum_{i=1}^\infty s_i^{-1} \E{|\Delta_i| \I{|\Delta_i| \geq \varepsilon s_i}} < \infty$ for all $\varepsilon > 0$,
\item [\textbf{L2.}] $\sum_{i=1}^\infty s_i^{-4} \E{\Delta_i^4 \I{|\Delta_i| \leq \delta s_i}} < \infty$ for some $\delta > 0$,
\item [\textbf{L3.}]  $\sum_{i=1}^\infty s_i^{-2}(\Delta_i^2 - \E{\Delta_i^2| \mathcal{G}_{i-1}})$ converges almost surely, and
\item [\textbf{L4.}]  $s_n^{-2} \sum_{i = n}^\infty \E{\Delta_i^2 | \mathcal{G}_{i-1}} \to 1$ almost surely,
\end{itemize}
then, almost surely,
\begin{align*}
\limsup_{n \to \infty}  \frac{Z_{n-1} - Z}{s_n  \sqrt{2 \log \log s_n^{-1}}}  = 1, \quad \liminf_{n \to \infty} \frac{Z_{n-1} - Z}{s_n  \sqrt{2 \log \log s_n^{-1}}}  = -1. 
\end{align*}
Finally, if, for all $p \in \N$ sufficiently large,  
\begin{itemize}
\item [\textbf{P1.}] $Z_n$ is bounded in $L_p$,
\item [\textbf{P2.}] $s_n^{-2p} \sum_{i=n}^\infty  \E{\Delta_i^{2p}} \to 0$,
\item [\textbf{P3.}] $s_n^{-2} \sum_{i = n}^\infty \E{\Delta_i^2 | \mathcal{G}_{i-1}}$ is bounded in $L_p$, 
\end{itemize}
and  \textbf{C1}, \textbf{C2} hold, then, the convergence in \eqref{conv1} is with respect to all moments.
\end{prop}
The central limit theorem and the law of the iterated logarithm in the proposition summarize Theorem 1(b), Corollary 1(b) and Corollary 2(b) in Heyde \cite{heyde77}.
The convergence of higher moments is an application of Theorem 1 in Hall \cite{hall78b}. To more precise, note that, for $p > 1$ sufficiently large, by \textbf{P1}, we can choose a sequence $k_n \uparrow \infty $ such that both $s_{n + k_{n} + 1} /s_{n+1} \to 0$ and $s_{n+1}^{-2p} \E{|Z_{n+ k_n} - Z|^{2p}} \to 0$. Further, we define $Z_{n,i} = (s_{n+1}^2 - s_{n + 1 + k_n}^2)^{-1/2} \sum_{j=1}^{i} \Delta_{n+j}$ and $ \mathcal{G}_{n,i} = \mathcal{G}_{n+i}$ for $n \geq 1, i \geq 0$. Then, $\{(Z_{n,i}, \mathcal{G}_{n,i} ), 0 \leq i \leq k_n\}$ is a zero-mean, square integrable martingale for all $n \geq 1$ and $\E{Z^2_{n,k_n}} = 1$. Moreover, conditions 
(4), (6), (9) in \cite[Theorem 1]{hall78b} follow from, in this order,  \textbf{C1},  \textbf{C2} and \textbf{P3}, \textbf{P2}. This theorem shows that $s_{n+1}^{-2p} \E{ \left| Z_{n+k_n} - Z_{n} \right| ^{2p}} \to \E{|\mathcal {N}|^{2p}}$ as $n \to \infty$. The claim now follows from our choice of $k_n$. 

In the literature about martingale limit theorems, results are often formulated in terms of an unconditional version of \textbf{L4} (or \textbf{C2}) : almost surely (in probability for \textbf{C2}),
\begin{align} \label{diffi}
s_n^{-2} \sum_{i = n}^\infty \Delta_i^2 \to 1.
\end{align}
The convergence \eqref{diffi} is at the very heart of Theorem 1 in \cite{heyde77}. It is easy to see and worked out in Lemma 1 in \cite{heyde77} that \textbf{L3} and \textbf{L4} imply \eqref{diffi}. 
In our work, as well as in the application given in \cite{heyde77} in the context of P{\'o}lya urns, it is considerably easier to verify conditions  \textbf{L3} and \textbf{L4} than establishing \eqref{diffi} directly.

Note the following improvements of the statements in Theorem \ref{thm_main}. First, by Theorem 1 in \cite{heyde77}, almost surely, the set of accumulation points of the sequence considered in the law of the iterated logarithm is $[-1, 1]$. 
Second, as usual in martingale central limit theorems, the convergence in \eqref{CLTtheorem} (or, more generally, in \eqref{conv1}) is mixing in the sense of R{\'e}nyi and R{\'e}v{\'e}sz \cite{renyi}: for a sequence of real-valued random variables $Y_n, n \geq 0,$ defined on a common probability space, and a random variable $Y$, we have $Y_n \to Y$ mixing, 
if for any random variable $X$ defined on the same probability space as $Y_n, n \geq 0,$ in distribution,
\begin{align*}
(Y_n, X) \to (Y,X^*) \quad  Y, X^* \: \text{independent}.
\end{align*} 
Here, $X^*$ has the same distribution as $X$.

\medskip \textsc{Tools in the proof - the profile polynomial.}
Among the conditions in  the martingale central limit theorem and the law of the iterated logarithm in Proposition \ref{prop:mart},  \textbf{L4} (or \textbf{C2}) is typically the hardest to verify. Here, we make use of the connection between the profile and the path length of the tree.
For $n, k \geq 0,$ we denote by $X_k(n)$ the number of nodes on level $k$ in $T_n$. We call $X_k(n), k \geq 0,$ the \emph{internal profile} of $T_n$. Given $\cF_{n}$, for $k \geq 0$, the probability that the $(n+1)$-st node is inserted on level $k$ is proportional to $U_k(n)$ defined by 
\begin{align} \label{defex}
U_k(n) = \begin{cases} \beta X_{k}(n) + m X_{k-1}(n) & 
\mbox{for }  n, k  \geq 1, \\
\I{k=n=0} & 
 \mbox{otherwise}.
\end{cases} \end{align}
We call $U_k(n), k \geq 0$, the \emph{external profile} of $T_n$, noting that, for $\beta$ integer, $U_k(n)$ counts the number of external nodes on level $k$ in $T_n$.
The profile polynomial $W_n(z)$ and its normalized version $M_n(z)$ were first introduced by Jabbour-Hattab in the BST \cite{jabbour},  
\begin{align}
W_n(z) = \sum_{k=0}^\infty U_k(n) z^k, \quad M_n(z) = \frac{W_n(z)}{\E{W_n(z)}}. \label{defWW}
\end{align}
Here, we let $z \in \C^+ = \{ z \in \C : \Re(z) > 0\}$. Note that, writing $W_n'(z) = \sum_{k=1}^\infty k U_k(n) z^{k-1}$ for the derivative of $W_n(z)$, we have  $W_n'(1) = (\beta + m)P_n + nm$ allowing to transfer results for the profile polynomial to the path length of $T_n$.
It is well-known that, for all $z \in \C^+$, the sequence $M_n(z)$ is a martingale with respect to the filtration $(\cF_n)$. The relevant term which has to be controlled in the verification of  \textbf{L4} turns out to be related to the second derivative $M_n''(1)$  (compare \eqref{conn} below), which converges by Weierstrass' convergence theorem for holomorphic functions upon verifying uniform almost sure convergence of $M_n(z)$ in a neighborhood of $z = 1$. This is the content of Proposition \ref{prop2}. Uniform almost sure convergence is related to uniform $L_2$-boundedness of $M_n(z)$ around $z=1$ which is known for all trees in our model. The generalization to uniform $L_p$-boundedness, $p > 2$, see Proposition \ref{prop3} below,  is at the core of Theorem  \ref{thm:moments} allowing for the verification of conditions \textbf{P1} and \textbf{P3}. It is this property which leads to the restriction to integer-valued $\beta$ in the theorem.

Uniform convergence of $M_n(z)$ also plays a decisive role in the analysis of both profiles $X_k(n), k \geq 0$ and $U_k(n), k \geq 0$ as $n \to \infty$ since the extraction of the coefficients of $W_n(z)$ is worked out with the help of Fourier's inversion formulae. We discuss relevant references in Section \ref{sec:poly}.

\section{Proof of Theorems \ref{thm_main} and \ref{thm:moments}}
\subsection{Preliminaries} \label{sec:pre}

We start with a basic result on $D_n$, defined as the depth of the $n$-th inserted node. In this context, we recall  the P{\'o}lya urn model with parameters $K \in \N, c_1, \ldots, c_K \in \N, (a_{i,j})_{1 \leq i, j \leq K}, a_{i,j} \in \N_0$. Here, starting with $c_i$ balls of color $i$, $1 \leq i \leq K$, in each step, one ball is drawn from the urn and replaced together with $a_{\ell, j}$ additional balls of each color $1 \leq j \leq K$, where $\ell$ denotes the color of the drawn ball. We say the urn has initial configuration $C = (c_1, \ldots, c_K)$ and replacement matrix $A = (a_{i,j})_{1 \leq i, j \leq K}$. For a general account on P{\'o}lya urns covering many of its applications in the analysis of 
random trees, we refer to Mahmoud's book \cite{Mah2008}.

The following lemma coincides with Theorem 7 in Dobrow and Smythe \cite{DoSm} in the case of integer-valued $\beta$. Here, and subsequently, for $n \geq 0$, we set 
\begin{align} \label{defa} \alpha_n :=  \sum_{k=0}^{\infty} U_k(n) = \begin{cases}(\beta + m) n - \beta
& \text{for }  n \geq 1, \\
1
&  \text{for }  n = 0.
\end{cases}. \end{align}
\begin{lem} For $1 \leq i < j$, denote by $A_{i,j}$ the event that the $i$-th inserted node is on the path from the root to the $j$-th inserted node. Then, for $n \geq 1$, the events
$A_{1,n}, \ldots, A_{n-1,n}$ are independent and $\Prob{A_{i,n}} = m/\alpha_i$.

\end{lem}
\begin{proof}
Let us first assume $\beta \in \N_0, m = 1$ or $\beta = -1, 1 < m \in \N$. Then, the number of external nodes in the subtree rooted at the $i$-th inserted node grows like the number of balls of color one in a P{\'o}lya urn
with $K=2$, initial configuration $C = (m, \alpha_i - m)$ and replacement matrix $A = ((\beta + m) \I{i =j})_{i,j = 1, 2}$.
In particular, the event $A_{i,j}$ for $j > i$ corresponds to drawing a ball of color one in the $(j-i)$-th step. It is well-known that the family of events $\{A_{i,j}, j > i \}$ is $i)$ independent of $\mathcal F_{i}$, and $ii)$ exchangeable. More precisely, for $0 \leq \ell \leq k$, and $\varepsilon_{i,j} \in \{0,1\}, i+1 \leq j \leq i+k,$ with $\sum_{j = i+1}^{i+k} \varepsilon_{i,j} = \ell$, the conditional probability
$$\Prob{\mathbf{1}_{A_{i, j}} = \varepsilon_{i, j} \ \text{for all} \ i+1 \leq j \leq i+ k \big | \mathcal F_{i}} = \frac{ \prod_{j = 1}^\ell (j (\beta + m) -\beta) \prod_{j = 1}^{\ell-k} 
(i + j -2) (\beta + m)}  {\prod_{j = 0}^{k-1} \alpha_{i + j}}$$
is deterministic and depends only on $\ell$ and $k$. The last expression and therefore $i)$ and $ii)$ remain true for arbitrary $\beta \in [0, \infty)$, $m = 1$.
In particular, for any set of parameters $\beta, m$ in our model, we have $\Prob{A_{i,j}} = m / \alpha_i$ for $j > i$. 
By the construction of the tree, for $1 \leq i < j \leq n-1$, the conditional probability of the event $A_{j,n}$ given $\mathcal F_{n-1}$ only depends on $N_{n-1}(j)$ defined as the size of the subtree rooted at the $j$-th inserted node at time $n-1$. Since, $N_{n-1}(j) = \sum_{k = j}^{n-1} \mathbf{1}_{A_{j,k}}$, using $i)$, $N_{n-1}(j)$ is independent of $\mathcal F_j$ and the same follows for $A_{j,n}$. In particular, the events $A_{i,j}$ and $A_{j,n}$ are independent.
It follows that $$\Prob{A_{i,n} \cap A_{j,n}} = \Prob{A_{i,j} \cap A_{j,n}} = \frac{m^2} {\alpha_i \alpha_j} = \Prob{A_{i,n}}\Prob{A_{j,n}}.$$
Hence, $A_{i,n}$ and $A_{j,n}$ are independent events. The case of more than two events follows from the same argument by a simple induction. 
\end{proof}
It follows that, with independent Bernoulli random variables $B_1, \ldots, B_{n-1}$ with $\E{B_i} = m/\alpha_i$, 
\begin{align}
 D_n  \stackrel{d}{=} \sum_{i=1}^{n-1} B_i, \quad \E{D_n}  =  \frac{m}{\beta + m} \log n +  O(1), \quad \text{Var}(D_n)  =  \frac{m}{\beta + m} \log n + O(1). \label{ding2} 
\end{align}
Since $P_n = \sum_{k=1}^n D_k$, using the first identity in the last display, we deduce
\begin{align} \label{ding4}
 \E{P_n}  = \frac{m \alpha_n}{\beta + m} \sum_{i = 1}^{n-1} \frac{1}{\alpha_i} - \frac{m(n-1)}{\beta + m}. 
\end{align}
The expansion \eqref{exp:mean} now follows from the facts that $\psi^{(0)}(x+1) =  \psi^{(0)}(x) +1/x, x > 0$ and $\psi^{(0)}(x) = \log x + O(1/x)$ as $x \to \infty$.
  
Below, we need a bound on the tails of $P_n - \E{P_n}$ as $n \to \infty$. To this end, we make use of a Chernoff-type bound for $D_n$. 
By an application of Bernstein's inequality for sums of uniformly bounded and independent random variables, for all $n \geq 1$, we have
\begin{align}
\Prob{|D_n - \E {D_n}| \geq t} \leq 2 e^{- t^2 / (2 \E{D_n} + t)}, \quad t > 0. \label{taildepth}
\end{align}
For the path length, the following rough bound will be sufficient in the remaining, 
\begin{align}
\Prob{|P_n - \E{P_n}| \geq t } & \leq \Prob{ \bigcup_{k=1}^n \{|D_k - \E{D_k}| \geq t / n\}}  \nonumber \\
& \leq n \sup_{1 \leq k \leq n} \Prob{|D_k - \E {D_k}| \geq t/n} \leq 2 n  e^{- t^2 / (2n^2 \E{D_n} + tn)}. \label{tail2}
\end{align}
Here we have used the fact that the right hand side of \eqref{taildepth} is increasing in $n$ for any fixed $t > 0$.
Note that the right hand side of the last display decays polynomially when $t  \sim \gamma \E{P_n}, \gamma > \gamma_0$, where $\gamma_0 > 0$ depends only on $\beta$ and $m$. This is far from optimal at least for integer-valued $\beta$ where it is well-known that
the left-hand side of \eqref{tail2} decays faster than $n^{-k}$ for any $k > 0$ when $t \sim \gamma \E{P_n}$ for any fixed $\gamma > 0$. (For a proof of this claim, we refer to \cite[Corollary 4.3] {Ro91} for the BST, to \cite[Theorem 5.1]{neiruesplit} and \cite[Corollary 1.2]{Munso} for $\beta = -1, 1 < m \in \N$, and to \cite[Theorem 5.6]{Munso2} for $\beta \in \N_0, m = 1$ from which the assertion can be deduced by checking the conditions formulated in the proof of Lemma 4.3 in \cite{neiruesplit}.) The best large deviation results for $P_n$ in the BST were proved by McDiarmid and Hayward \cite{MCDHAY}.

We turn to the external profile $U_k(n)$ and the external path length defined by $E_n = \sum_{k=1}^n k U_k(n)$. For $\beta$ integer, $E_n$ equals the sum of the depths of the external nodes in $T_n$. 
From the construction of the tree it follows that, with $E_0 := 0$, 
\begin{align}
E_n  = E_{n-1} + (\beta  + m) D_n + m = (\beta + m)P_n + nm  , \quad \E{D_n | \cF_{n-1}} = \frac{E_{n-1}}{ \alpha_{n-1}}, \quad n \geq 1. \label{ding1} 
\end{align}
Using these two identities, one can check that
\begin{align}
\text{Var}(P_n) & = \text{Var}(P_{n-1}) \left(1 + \frac{2(\beta + m)} { \alpha_{n-1}} \right)+ \text{Var}(D_n) \nonumber \\
& = \sum_{i=1}^n \text{Var}(D_i) \prod_{j=i+1}^n \left(1 + \frac{2(\beta + m)} { \alpha_{j-1}} \right) \nonumber \\
& =  \left(n - \frac{\beta}{\beta + m} \right) \left(n + \frac{m}{\beta + m} \right) \sum_{i=1}^n \frac{\text{Var}(D_i)}{ (i - \beta / (\beta + m)) (i - \beta / (\beta + m) + 1)}. \label{ding3}
\end{align}
By the expansion of the variance in \eqref{ding2}, it immediately follows that the sequence $S_n$ is $L_2$-bounded. Using the last display together with the more precise first identity in \eqref{ding2}, a straightforward but lengthy calculation leads to \eqref{exactvar}.

We continue by collecting some immediate consequences for the sequence of increments $X_n :=  S_n - S_{n-1}$, $n \geq 1$.
Abbreviating $\mu_n := \E{E_n}$, by definition and \eqref{ding1}, 
\begin{align}
X_n & = \frac{E_n - \mu_n}{\alpha_n} - \frac{E_{n-1} - \mu_{n-1}}{ \alpha_{n-1}}  =   (\beta + m) \frac{ \alpha_{n-1} D_n - E_{n-1}}{\alpha_n \alpha_{n-1}} \nonumber \\
 & =  \frac{\beta + m}{\alpha_n}  \left(  D_n - \E{D_n} - S_{n-1}  \right).  \label{X1}
\end{align}
The martingale property of $S_n$ follows from the last line, since, by \eqref{ding1}, we have $\E{D_n - \E{D_n} | \cF_{n-1}} = S_{n-1}$. As $S_n$ is $L_2$-bounded there exists a random variable $S$ with $S_n \to S$ almost surely and in $L_2$. Using \eqref{X1}, \eqref{ding1} and \eqref{ding2}, we have
\begin{align*}
\E{X_n^2} &  =  \frac{(\beta + m)^2}{\alpha_n^2} \left( \text{Var}(D_n) -  \E{S^2}(1+(o(1)) \right) = \frac{m}{\beta + m} \frac{\log n}{n^2} +  O\left (\frac{1}{ n^2}\right).
\end{align*}
It follows 
\begin{align} \label{evaint}
s_n^2 = \sum_{i=n}^\infty \E{X_i^2} = \frac{m}{\beta + m} \int_{n}^\infty \frac{\log x}{x^2} dx + O \left ( \frac{1}{n} \right) = \frac{m}{\beta + m} \frac{\log n}{n} + O \left ( \frac{1}{n} \right), \end{align}
and \begin{align} \label{approx} s_n^{-2} = \frac{\beta + m}{m} \frac{n}{\log n} + O \left( \frac{n}{\log^2 n} \right). \end{align}

\subsection{The profile polynomial} \label{sec:poly}
Recall the definition of the profile polynomial $W_n(z)$ and its normalized version $M_n(z)$ in \eqref{defWW}. From the definition of the external profile \eqref{defex}, it follows that
$U_k(n)  = U_k(n-1) + \beta \mathbf{1}_{\{D_n = k\}} + m  \mathbf{1}_{\{D_n = k-1\}}$ for $n, k \geq 1$. Hence, 
$$\E{U_k(n) | \cF_{n-1}} = \frac{(\beta + \alpha_{n-1}) U_k(n-1) + m U_{k-1}(n-1)}{\alpha_{n-1}}, \quad n, k  \geq 1.$$
For $z \in \C^+$, we deduce $W_1(z) = mz$, and, for $n \geq 2$,
$$\E{W_n(z) | \cF_{n-1}} = \frac{\alpha_{n-1} + \beta + mz}{\alpha_{n-1}} W_{n-1}(z).$$
The martingale property of $M_n(z), z \in \C^+,$ follows readily. 
Since $W_n(z)$ is a polynomial of degree at most $n$ and $C_n(z)$ is a polynomial of degree $n$, both functions are holomorphic on $\C^+$. Indicating the first and second derivative of a complex function $f$ by $f'$ and $f''$, an application of the Differentiation lemma, see e.g.\ Klenke \cite[Theorem 6.28] {klenke}, shows that $C_n'(z) = \E{W_n'(z)}$ and $C_n''(z) = \E{W_n''(z)}$. From the definition of $W_n(z)$ in \eqref{defWW}, it follows immediately that   
\begin{align}
W_n(1) & = \alpha_n, \quad W_n'(1)  = E_n, \quad W_n''(1) = \alpha_n \E{D_{n+1}^2 | \mathcal{F}_n} - E_n, \label{wn1} \\
C_n(1) & = \alpha_n, \quad C_n'(1)  = \mu_n, \quad C_n''(1) = \alpha_n \E{D_{n+1}^2} - \mu_n.  \label{cn1}
\end{align}
Moreover, we have $S_n = M_n'(1)$.
\begin{prop} \label{prop2}
There exists a neighborhood of $z=1$ in the complex plane where, almost surely, $M_n(z)$ converges uniformly to a limit denoted by $M(z)$.
\end{prop}
\begin{prop} \label{prop3}
Let either $\beta \in \N_0, m =1$ or $\beta = -1, 1 < m \in \N$. Then, for any $p > 0$, the sequence $M_n(z)$ is uniformly bounded in $L_p$ in a neighborhood of $z = 1$. In other words, there exists $ \varepsilon > 0$ and $C > 0$ such that, for all $z \in \C, |z-1| < \varepsilon, n \geq 0,$ we have $\E{|M_n(z)|^p} \leq C$. The same is true for all derivatives of $M_n(z)$.
\end{prop}

There are two different approaches in the literature towards uniform almost sure convergence of $M_n$, both of which were established initially in the BST. First, the work of Chauvin, Drmota and Jabbour-Hattab \cite{chdrja} is based on explicit calculations involving the covariance function of $W_n(z)$ and uses a limit theorem for vectorial martingales due to Neveu \cite{neveu}. Katona \cite{katona} generalized the methodology to 
the class of linear recursive trees considered here where $\beta \in [0, \infty), m = 1$. (To be precise, he considers a model where a node $v$ is chosen as parent of the new node to be inserted with probability $d_v + 1 + \alpha$, $\alpha > -1$. This is equivalent to our model when $\beta = (1+\alpha)^{-1}$.)
The second approach goes back to Chauvin et al. \cite{chklmaro} and relies on an embedding of $T_n$ into continuous time. It improves on the results in \cite{chdrja} providing optimal ranges of the complex plane where $M_n(z)$ converges uniformly almost surely and in $L_p, p > 1$. It is based on Biggins' uniform convergence results in branching random walks \cite{biggins}.
Schopp \cite{schoppprofile} generalized the approach to a wider class of trees covering $m$-ary trees and linear recursive trees with $\beta$ integer. Summarizing, Proposition \ref{prop2} follows from Corollary 3 in \cite{katona} together with Theorem 5.5 in \cite{schoppprofile}.

\medskip In order to prove Proposition \ref{prop3}, we  recall the details about the embedding of $T_n$ into a continuous-time model. We only treat the case $\beta \in \N_0, m = 1$, the case of $\beta = -1$ being easier. Consider a continuous-time branching random walk describing the evolution of a population where, for any $t \geq 0$, each individual is assigned a position on the positive real line  with the following dynamics.
At time $t = 0$, there is one alive individual at position $x = 0$. 
Individuals do not move throughout their lifetimes and die at unit rate independently of each other. Extinction of an individual at position $x \geq 0$, instantaneously gives rise to the birth of  $\beta +2$ new individuals, $\beta + 1$ of which at position $x$, and one at position $x +1$. 
For $t \geq 0, k \in \N_0$, denote by $\varrho_t(k)$ the number of alive individuals at position $k$. For $n \geq 1$, denote by $\tau_n$ the time of the $n$-th death event. We abbreviate $\tau_0 = 0$. At time $t$ with $\tau_n \leq t < \tau_{n+1}$, we have $\alpha_{n+1}$ alive individuals in the population, where $\alpha_n$ is defined in \eqref{defa}.
 Since any alive individual is equally likely to die next, for all $n \geq 0$, alive individuals at position $k \geq 0$ in the branching random walk at time $\tau_n$ can be identified with external nodes on level $k+1$ in the tree $T_{n+1}$. More formally, recalling \eqref{defex}, we have $(\varrho_{\tau_n}(k), k \geq 0) = (U_{k+1}(n+1), k \geq 0)$, where equality can be understood in an almost sure sense even for the sequences indexed by $n \geq 0$ upon choosing suitable versions of the processes. For our purposes, it is sufficient that equality holds in distribution for fixed $n$.
Analogously to the profile polynomial, for $z \in \C^+$,  we define
$$\mathcal W_t(z) = \sum_{k \geq 0} \varrho_t(k) z^{k}, \quad \mathcal M_t(z) = \frac{\mathcal W_t(z)}{\E{\mathcal W_t(z)}}. $$
Let $\mathcal F^*_t, t > 0$ be the $\sigma$-field containing the information of the branching random walk up to time $t$. Then, it is well-known that $\mathcal M_t(z)$ is a martingale, often referred to as Biggins' martingale \cite[Section 5]{biggins}. (Note that, the process $\mathcal M_t(z)$ equals $W^{(t)}(-\text{Log z})$ in the notation of \cite{biggins}, where 
$\text{Log z}$ denotes the principal value of the logarithm of $z$.) Moreover, setting $\mathcal C_t(z) = \E{\mathcal W_t(z)}$, we have
\begin{align} \label{eq20}
\mathcal M_{\tau_n}(z) = \frac 1 z \frac{C_{n+1}(z)}{\mathcal C_{\tau_n}(z)} M_{n+1}(z),
\end{align}
where $\mathcal C_{\tau_n}(z)$ and $M_{n+1}(z)$ are independent. This follows since the skeleton of $\mathcal M_t(z)$ is independent of the jump times $\tau_n, n \geq 1$. Again, equality can be understood either on an almost sure level or in distribution. 

\begin{proof} [Proof of Proposition \ref{prop3}]

From the continuous-time analogue of Proposition 6.1 in \cite{grza}, compare Remark 6.2 there, we know that, for any $p > 0$,  there exists a neighborhood $O$ of $z = 1$ and a constant $C_1 > 0$ such that $\E{\mathcal |\mathcal M_t(z)|^p} \leq C_1$ for all $t \geq 0, z \in O$. In particular, for $z \in O$, the sequence $\mathcal M_t(z), t > 0$ is uniformly integrable. Writing $\mathcal M_\infty(z)$ for the almost sure limit, we have $ \mathcal M_{\tau_n}(z) = \E{\mathcal M_\infty(z) | \mathcal F^*_{\tau_n}}$ almost surely by the theorem of optional sampling \cite[Theorem 77.5]{rogwil}. Using the triangle and Jensen's inequality, we deduce that, for any $z \in O$ and $n \geq 1$, 
\begin{align*} \E{|\mathcal M_{\tau_n}(z)|^p} = \E{ \left| \E{\mathcal M_{\infty}(z)| \mathcal F^*_{\tau_n}} \right|^p}  \leq \E{|\mathcal M_{\infty}(z)|^p} = \lim_{t \to \infty} \E{|\mathcal M_{t}(z)|^p}  \leq C_1.  \end{align*}

Let $H_n(z) := C_{n+1}(z) / (z\mathcal{C}_{\tau_n}(z))$ and observe that, by independence and since both $M_{\tau_n}(z)$ and $M_{n+1}(z)$ have unit mean, \eqref{eq20} implies $\E{H_n(z)} = 1$ for all $z \in O$. For $p \geq 1$, by Jensen's and the triangle inequality, we obtain $\E{|H_n(z)|^p} \geq |\E{H_n(z)}|^p = 1$. Hence, again using  \eqref{eq20}, for $z \in O$, we obtain $\E{|M_{n+1}(z)|^p} \leq \E{|\mathcal M_{\tau_n}(z)|^p} \leq C_1$. Thus, for any $p > 0$, $M_n(z)$ is bounded uniformly in $L^p$ in a neighborhood of $z = 1$.
The result transfers to the derivatives of $M_n(z)$ by a standard application of Cauchy's integral formulae, compare, e.g.\ the proof of Proposition 5.8 in \cite{grza}. (Note that, as worked out in \cite[Section 5]{schoppprofile}, for any $z \in \mathbb C^+$, the sequence $H_n(z)$ is a martingale with respect to the filtration $( \mathcal F^*_{\tau_n})$ that converges almost surely. For fixed $z \in \mathbb C^+$, its limit is distributed like a multiple of a complex power of a random variable with a Gamma distribution. None of these deeper results are needed here.) 
\end{proof}

We decided to prove the proposition with the help of martingale theory in order to stay in the framework of this work. An alternative proof could be given relying on a distributional recurrence for the sequence $W_n(z)$. This approach was taken by Drmota, Janson and Neininger \cite{drjane} studying the profile of $m$-ary search trees including the case of BSTs. As for the martingale approach, this methodology can be worked out for any $ \beta \in \N_0$ or $\beta = -1$.

\medskip In order to prove Theorem 1 for non-integer valued $\beta$, the following (presumably non-optimal) result is sufficient. 
\begin{prop} \label{prop1}
For all $p \in \N, \varepsilon > 0$, we have $\E{|S_n|^p} = O(n^{\varepsilon})$.
\end{prop}
\begin{proof}
By Jensen's inequality it is sufficient to consider even values of $p$. We proceed by induction. 
Let $p \geq 4$ and assume that, for some $C_1 > 0$, we have $\E{(E_{n} - \E{E_{n}})^{p-2}} \leq C_1 n^{p-2 + \varepsilon}$ for all $n \geq 1$.
By \eqref{ding1}, we have
\begin{align}
\E{(E_n - \E{E_n})^p} & = \E{(E_{n-1} - \E{E_{n-1}})^p} + p \E{(E_{n-1} - \E{E_{n-1}})^{p-1} (D_n - \E{D_{n})}} (\beta + m) \nonumber \\
& \: \: \: \: + \sum_{k=2}^p {p \choose k} \E{(E_{n-1} - \E{E_{n-1}})^{p-k} (D_n - \E{D_{n}})^k} (\beta + m)^k\label{eq7}
\end{align}
From the representation \eqref{ding2} for $D_n$, it is easy to see that, for all $p > 0$, we have
\begin{align} \label{moments:Dn}
\E{|D_n - \E{D_n}|^p} = O( (\log n)^p).
\end{align}
Thus, by an application of H{\"o}lder's inequality, there exist  constants $C_2, C_3 $ such that
\begin{align*} \frac{p(p-1)(\beta + m)^2}{2} &  \E{(E_{n-1} - \E{E_{n-1}})^{p-2}  (D_n - \E{D_{n}})^2}  \\ & \leq C_2  \E{(E_{n-1} - \E{E_{n-1}})^{p}}^{(p-2) / p} (\log n)^2 \end{align*}
and 
\begin{align*}
\sum_{k=3}^p {p \choose k} & \E{(E_{n-1} - \E{E_{n-1}})^{p-k} (D_n - \E{D_{n}})^k} (\beta + m)^k\\
& \leq C_3 \sum_{k=3}^p \E{(E_{n-1} - \E{E_{n-1}})^{p-2}}^{(p-k)/(p-2)} (\log n)^k.
\end{align*}
Using \eqref{ding1} to simplify the second summand in \eqref{eq7}, the bounds obtained above and the induction hypothesis, we summarize
\begin{align*}
\E{(E_n - \E{E_n})^p} & \leq \left(1 + \frac{p(\beta + m)}{\alpha_{n-1}}\right) \E{(E_{n-1} - \E{E_{n-1}})^p}  \\ 
& + C_2  \E{(E_{n-1} - \E{E_{n-1}})^{p}}^{(p-2) / p} (\log n)^2+ C_1 C_3 \sum_{k=3}^p n^{p-k + \varepsilon} (\log n)^k.
\end{align*}
We now proceed by induction over $n$. Assume that $ \E{(E_{i} - \E{E_{i}})^p} \leq C_4 i^{p + \varepsilon}$ for some constant $C_4 \geq 1$ and all $i \leq n-1$.
Then, for some constant $C_5$ depending only on $\beta, m, p, \varepsilon$ but not on $C_4$ or $n$, 
\begin{align*}
\E{(E_n - \E{E_n})^p} & \leq C_4 n^{p + \varepsilon} \left( \left(1 + \frac{p(\beta + m)}{\alpha_{n-1}}\right)\left( 1- \frac{1}{n}\right)^{p+\varepsilon} + C_2 n^{-2}(\log n)^2 \right. \\
& \: \: \: \: + C_1 C_3  p n^{-3}(\log n)^p\Bigg)  \\
& \leq C_4 n^{p + \varepsilon}  \left( 1 + \frac{p(\beta + m)}{\alpha_{n-1}} - \frac {p + \varepsilon}{n}  + C_5 n^{-2} (\log n)^2\right). 
\end{align*}
The bracket on the  right hand side of the last display does not exceed one for all $n$ large enough. This concludes the proof.
\end{proof}
The proposition leads to a bound on the moments of the martingale differences $X_n$.
For any $p \geq 2$, by \eqref{X1}, using $(a + b)^p \leq 2^p (a^p + b^p)$ for $a, b \geq 0$, we have
\begin{align*}
|X_n|^p \leq \frac{2^p (\beta + m)^p}{\alpha_n^p} \left( |S_{n-1}|^p +  |D_{n} - \E{D_{n}}|^{p} \right).
\end{align*}
Now from Proposition \ref{prop1} and \eqref{moments:Dn}, for any $\varepsilon > 0$, we deduce 
\begin{align} \label{bound_moments} \E{|X_n|^p} =  O\left(n^{\varepsilon - p} \right). \end{align}

\subsection{Verifications of the conditions in Proposition \ref{prop:mart}}
We start with the conditions on the moments.  For any $\delta> 0, 0 < \varepsilon <1$, using \eqref{bound_moments}, for some $C > 0$, we have
\begin{align} \label{L1}
\sum_{i=1}^\infty s_i^{-4} \E{X_i^4 \I{|X_i| \leq \delta s_i}} \leq \sum_{i=1}^\infty s_i^{-4} \E{X_i^4} \leq C \sum_{i=1}^\infty i^{\varepsilon - 2} < \infty.
\end{align}
This shows \textbf{L2}. In the same way, one verifies \textbf{P2}.
In the verification of conditions \textbf{C1} and \textbf{L1} we make use of the tail bounds \eqref{taildepth} and \eqref{tail2}.  Let $\varepsilon > 0$. First, by the Cauchy-Schwarz inequality, 
\begin{align*}
s_n^{-2} \sum_{i=n}^\infty  \E{X_i^2 \I{|X_i| \geq \varepsilon s_n}}  \leq s_n^{-2} \sum_{i=1}^\infty  \left(\E{X_i^4}\right)^{1/2} \left(\Prob{|X_i| \geq  \varepsilon s_n}\right)^{1/2}
\end{align*}
Using \eqref{bound_moments}, \eqref{taildepth} and\eqref{tail2}, it is easy to see that there exist constants $c, C > 0$ such that
\begin{align*}
s_n^{-2} \sum_{i=n}^\infty  \E{X_i^2 \I{|X_i| \geq \varepsilon s_n}} & \leq C \frac{\log n}{n} \sum_{i=n}^\infty   e^{-c i \sqrt{(\log n) / n}} \nonumber \\ 
& \leq C \frac{\log n}{n} \left( \int_{x=n}^\infty   e^{-c x \sqrt{(\log n) / n}} dx + e^{-c \sqrt{n \log n}}\right) \\
& = O\left( \sqrt{\frac{\log n}{n}} e^{-c \sqrt{n\log n}}\right).
\end{align*}
This proves condition $\textbf{C1}$. The verification of $\textbf{L1}$ runs along similar lines.
\noindent In order to prove \textbf{L3}, note that, by Proposition IV-6-1 in Neveu \cite{neveuenglish},  the series converges almost surely if  \begin{align} \label{helpL2}
\sum_{i=1}^\infty s_i^{-4} \E{ (X_i^2 - \E{X_i^2| \mathcal{F}_{i-1}})^2} < \infty. \end{align}
By Jensen's inequality, 
\begin{align*}
\E{ (X_n^2 - \E{X_n^2| \mathcal{F}_{n-1}})^2} & \leq \E{X_n^4} + \E{\E{X_n^2 | \mathcal{F}_{n-1}}^2}  \leq 2 \E{X_n^4}. \end{align*}
Thus, \eqref{helpL2} follows from \eqref{L1} proving \textbf{L3}. 

We move on to condition \textbf{L4} also covering \textbf{C2} proving the martingale central limit theorem and the law of the iterated logarithm.
To this end, by \eqref{X1},
\begin{align}
\frac{\alpha_n^2}{(\beta + m)^{2}} \E{X_n^2 | \mathcal{F}_{n-1}} & = \frac{\E{(E_{n-1} - \alpha_{n-1}D_n)^2 | \mathcal{F}_{n-1}}}{\alpha_{n-1}^2}  \nonumber\\
& = \frac{\alpha_{n-1}^2 \E{D_n^2 | \mathcal{F}_{n-1}}- E^2_{n-1} }{\alpha_{n-1}^2  } \nonumber \\
& =  \E{D_n^2 | \mathcal{F}_{n-1}} - \E{D_n}^2 - 2 S_{n-1} \E{D_n} - S_{n-1}^2. \label{xxx}
\end{align}
In order to analyze this expression, we make use of the profile polynomial defined in \eqref{defWW}.
By the product rule, 
\begin{align*}
M_n''(z) & 
= \frac{(W_n''(z) C_n(z)  - C_n''(z) W_n(z))C_n(z) - 2 C_n'(z)(W_n'(z) C_n(z) - C_n'(z) W_n(z))}{C_n^3(z)}.
\end{align*}
Using \eqref{wn1} and \eqref{cn1}, we deduce
\begin{align*}
M_n''(1) &  = \E{D_{n+1}^2 | \mathcal{F}_n}  - \E{D_{n+1}^2}  -  \frac{(2 \mu_n +\alpha_n )(E_n - \mu_n)}{\alpha_n^2} \\
& = \E{D_{n+1}^2 | \mathcal{F}_n}  - \E{D_{n+1}}^2 - \text{Var}(D_{n+1}) - 2 \E{D_{n+1}} S_n - S_n. 
\end{align*}
Together with \eqref{xxx}, it follows
\begin{align} \label{conn}
(\beta + m)^{-2} \alpha_n^2\E{X_n^2 | \mathcal{F}_{n-1}} & =  \text{Var}(D_n) + M_{n-1}''(1) + S_{n-1} - S_{n-1}^2.
\end{align}
Hence, recalling \eqref{approx}, we have 
\begin{align*}
s_n^{-2} & \sum_{i = n}^\infty \E{X_i^2 | \mathcal{F}_{i-1}}  \\
& = \left( \frac{(\beta + m)^3}{m} \frac{n}{\log n} + O \left( \frac{n}{\log^2 n} \right) \right) \sum_{i=n}^\infty \frac{ \frac{m}{\beta + m}  \log i + M_{i-1}''(1) +  S_{i-1} - S_{i-1}^2}{\alpha_i^2}  \to 1,
\end{align*}
almost surely, by computing the series as in \eqref{evaint}. Here, we use that, almost surely,  $S_n \to S$ and $M''_{n}(1) \to M''(1)$ by Weierstrass' convergence theorem for holomorphic functions recalling Proposition \ref{prop2}.
This verifies condition \textbf{L4} and finishes the proof of Theorem \ref{thm_main}. In order to conclude the proof of Theorem \ref{thm:moments} note that condition \textbf{P1} follows immediately from Proposition \ref{prop3} and
$S_n = M_n'(1)$. For a real-valued random variable $Y$ and $p \geq 1$, denote by $\| Y \|_p := \E{|Y|^p}^{1/p}, p > 1$ the $L_p$-norm $Y$. For $p \geq 2$, we have 
\begin{align*}
& \left \| (\beta + m)^{-2} \sum_{i = n}^\infty \E{X_i^2 | \mathcal{F}_{i-1}} \right\|_p  \\
& \: \leq \sum_{i = n}^\infty \frac{\text{Var}(D_i)} {\alpha_i^2} + \sum_{i = n}^\infty \frac{ \| M_{i-1}''(1) \|_p +  \|S_{i-1}\|_p + \|S_{i-1}^2\|_p}{\alpha_i^2},
\end{align*}
where we applied the Minkowski inequality. The term on the right hand side is bounded from above by $C (\log n) / n$ for some $C > 0$ by Proposition \ref{prop3} and \eqref{ding2}. Together with \eqref{approx}, this implies  \textbf{P3} and concludes the proof.

\small
\setlength{\bibsep}{0.3em}

\end{document}